\documentclass{amsart}
\usepackage{amssymb,amscd}
\DeclareMathSymbol{\twoheadrightarrow} {\mathrel}{AMSa}{"10}

\def\Q{{\mathbf Q}}
\def\Z{{\mathbf Z}}

\def\F{{\mathbf F}}
                               \def\ST{{\mathbf S}}

\def\Gal{\mathrm{Gal}}

\def\End{\mathrm{End}}
\def\Aut{\mathrm{Aut}}
\def\Hom{\mathrm{Hom}}

\def\fchar{\mathrm{char}}

\def\GL{\mathrm{GL}}

\def\ZZ{\mathcal{Z}}

\def\dim{\mathrm{dim}}

\newtheorem{thm}{Theorem}[section]
\newtheorem{lem}[thm]{Lemma}
\newtheorem{cor}[thm]{Corollary}

\theoremstyle{definition}

\newtheorem{ex}[thm]{Example}

           \newtheorem{rem}[thm]{Remark}

\title[Homomorphisms of abelian varieties over function fields]
{Homomorphisms of abelian varieties over geometric fields of finite
characteristic}
\author[Yuri G. Zarhin]{Yuri G. Zarhin}
\address{Department of Mathematics, Pennsylvania State University,
University Park, PA 16802, USA}

 \email{zarhin\char`\@math.psu.edu}
\begin{document}
\begin{abstract}
We study  analogues of Tate's conjecture on homomorphisms for
abelian varieties  when the ground field is finitely generated over
an algebraic closure of a finite field. Our  results cover the case
of abelian varieties without nontrivial endomorphisms.
\end{abstract}

\maketitle
\section{Introduction}
Let $K$ be a field, $\bar{K}$ its algebraic closure, $\bar{K}_s
\subset \bar{K}$ the separable algebraic closure of $K$,
$\Gal(K)=\Gal(\bar{K}_s/K)=\Aut(\bar{K}_s/K)$ the absolute Galois
group of $K$. Let $X$ be an abelian variety over $K$. Then we write
$\End_K(X)$ for its ring of $K$-endomorphisms and put
$\End_K^0(X):=\End_K(X)\otimes\Q$. We write $\End(X)$ for the ring
$\End_{\bar{K}}(X)$ of $\bar{K}$-endomorphisms of $X$ and write
$\End^0(X)$ for the corresponding finite-dimensional semisimple
$\Q$-algebra $\End(X)\otimes\Q$. If $Y$ is an abelian variety over
$K$ then we write $\Hom_K(X,Y)$ for the (free) commutative group of
$K$-homomorphisms from $X$ to $Y$.

If $n$ is a positive integer that is not divisible by $\fchar(K)$
then we write $X_n$ for the kernel of multiplication by $n$ in
$X(\bar{K})$; it is well known that $X_n$ is free $\Z/n\Z$-module of
rank $2\dim(X)$ \cite{Mumford}, which is a Galois submodule of
$X(\bar{K}_s)$. We write $\bar{\rho}_{n,X}$ for the corresponding
(continuous) structure homomorphism
$$\bar{\rho}_{n,X}: \Gal(K)\to \Aut_{\Z/n\Z}(X_n) \cong \GL(2\dim(X),\Z/n\Z).$$
In particular, if $n=\ell$ is a prime then $X_{\ell}$ is a
$2\dim(X)$-dimensional $\F_{\ell}$-vector space provided with
$$\bar{\rho}_{\ell,X}:\Gal(K)\to \Aut_{\F_{\ell}}(X_{\ell})\cong
\GL(2\dim(X),\F_{\ell}).$$

If $\ell$ is a prime that is different from $\fchar(K)$ then we
write $T_{\ell}(X)$ for the $\Z_{\ell}$-Tate module of $X$ and
$V_{\ell}(X)$ for the corresponding $\Q_{\ell}$-vector space
$$V_{\ell}(X)=T_{\ell}(X)\otimes_{\Z_{\ell}}\Q_{\ell}$$
provided with the natural continuous Galois action \cite{Serre}
$$\rho_{\ell,X}:\Gal(K) \to \Aut_{\Z_{\ell}}(T_{\ell}(X))\subset
\Aut_{\Q_{\ell}}(V_{\ell}(X)).$$ Recall \cite{Mumford} that
$T_{\ell}(X)$ is a free $\Z_{\ell}$-module of rank $2\dim(X)$ and
$V_{\ell}(X)$ is a $\Q_{\ell}$-vector space of dimension $2\dim(X)$.
Notice that there are canonical isomorphisms of $\Gal(K)$-modules
$$X_{\ell}=T_{\ell}(X)/\ell T_{\ell}(X) \eqno(0).$$

 There are natural algebra
injections
$$\End_K(X)\otimes \Z/n \hookrightarrow \End_{\Gal(K)}(X_n) \eqno(1),$$
$$\End_K(X)\otimes \Z_{\ell} \hookrightarrow \End_{\Gal(K)}(T_{\ell}(X))\eqno(2),$$
$$\End_K(X)\otimes \Q_{\ell} \hookrightarrow \End_{\Gal(K)}(V_{\ell}(X))\eqno(3).$$
It is known \cite[Sect. 1]{Tate} that for given $\ell,K,X,Y$ the map
in (2) is bijective if and only if the map in (3) is bijective.

The Tate conjecture on homomorphisms of abelian varieties
\cite{Tate} asserts that if $K$ is finitely generated over its prime
subfield then the last two injections are bijective.  This
conjecture was proven by J. Tate himself over finite fields
\cite{Tate},  the author \ when $\fchar(K)>2$
\cite{ZarhinIz,ZarhinMZ1}, G. Faltings when $\fchar(K)=0$
\cite{Faltings1,Faltings2} and by S. Mori when $\fchar(K)=2$
\cite{MB}. They also proved (in the corresponding characteristics)
that the Galois module $V_{\ell}(X)$ is semisimple. (In the case of
finite fields the semisimplicity result is due to A. Weil. See also
\cite{ZarhinP}.) Let us state explicitly the following two well
known  corollaries of the Tate conjecture. (Here we assume that $K$
is finitely generated over its prime subfield.)

\begin{itemize}
\item[(i)]
{\bf The  isogeny theorem}. If for some $\ell \ne \fchar(K)$ the
$\Gal(K)$-modules $V_{\ell}(X)$ and $V_{\ell}(Y)$ are isomorphic
then $X$ and $Y$ are isogenous over $K$. (See \cite[Sect. 3, Th.
1(b) and its proof]{Tate} and \cite[Proof of Cor. 1.3 on p.
118]{Schappacher}.)
\item[(ii)]
If $\End_K(X)=\Z$ then the $\Gal(K)$-module $V_{\ell}(X)$ is
absolutely simple.
\end{itemize}

In addition,  if $K$ is finitely generated over its prime subfield
and $\fchar(K)\ne 2$ then for all but finitely many primes $\ell$
the $\Gal(K)$-module $X_{\ell}$ is semisimple and the injection
$$\Hom_K(X,Y)\otimes \Z/\ell \hookrightarrow \Hom_{\Gal(K)}(X_{\ell},Y_{\ell})$$
in (1) is bijective (\cite[Th. 1.1]{ZarhinMZ2},\cite[Cor. 5.4.3 and
Cor. 5.4.5]{ZarhinIn}, \cite[Prop. 3.4]{SZ},
\cite[Th.4.4]{ZarhinMatSb}).  (See \cite[Cor. 10.1]{ZarhinG} for a
discussion of the case of finite fields.) It follows immediately
that if $\End_K(X)=\Z$ then for all but finitely many primes $\ell$
the Galois module $X_{\ell}$ is absolutely simple. We discuss an
analogue of the isogeny theorem with ``finite coefficients" in
Section \ref{isogenyF}.

Let $p$ be a prime, $\F$ a finite field of characteristic $p$ and
$\bar{\F}$  an algebraic closure of $\F$.
 The aim of this note is to discuss the situation when the ground field $L$ is a
 field of characteristic $p$ that (strictly) contains $\bar{\F}$ and is finitely
 generated over it. We call such a field a {\sl geometric field} of
 characteristic $p$. Geometric fields are precisely the fields of
 rational functions of irreducible algebraic varieties (of positive dimension) over
 $\bar{\F}$.

 Our main results are the following four theorems.

 \begin{thm}
 \label{main}
 Let $p>2$ be a prime, $L$ a {\sl geometric field} of
 characteristic $p$ and $X$  an abelian variety  of positive dimension over $L$.
 Suppose that $\End_{L}(X)=\Z$. Then:

 \begin{itemize}
 \item[(i)]
For all primes $\ell \ne \fchar(L)$ the Galois module $V_{\ell}(X)$
is absolutely simple.
\item[(ii)]
For all but finitely many primes $\ell$ the Galois module $X_{\ell}$
is absolutely simple.
 \end{itemize}
 \end{thm}

\begin{rem}
In the case of  $\End(X)=\Z$ the assertion (i) of Theorem \ref{main}
follows from \cite[Cor. 1.4]{ZarhinMZ3}.
\end{rem}

\begin{rem}
Theorem \ref{main} gives a positive answer to a question of W. Gajda
that was asked in connection with \cite{Gajda}.

\end{rem}

\begin{thm}
 \label{isogeny}
Let $p>2$ be a prime, $L$ a {\sl geometric field} of
 characteristic $p$ and $X$ and $Y$ are abelian varieties  of positive dimension over $L$.
 Suppose that $\End_{L}(X)=\Z$ and one of the following two
 conditions holds:

 \begin{itemize}
\item[(i)]
There exists a prime $\ell$ such that the $\Gal(L)$-modules
$V_{\ell}(X)$ and $V_{\ell}(Y)$ are isomorphic.
\item[(ii)] The $\Gal(L)$-modules $X_{\ell}$ and $Y_{\ell}$ are isomorphic for infinitely many
primes $\ell$.
 \end{itemize}
 Then $X$ and $Y$ are isogenous over $L$.
 \end{thm}

\begin{rem}
There are plenty of explicit examples in characteristic $p>2$ of
abelian varieties $X$ with $\End(X)=\Z$
\cite{ZarhinMRL2,ZarhinBSMF}.
\end{rem}

\begin{thm}
 \label{center}
 Let $p>2$ be a prime, $L$ a {\sl geometric field} of
 characteristic $p$ and $X$  an abelian variety  of positive dimension over
 $L$. Let $\ZZ$ be the center of $\End_L(X)$.
 Then:

 \begin{itemize}
 \item[(i)]
For all primes $\ell \ne \fchar(L)$ the center $\ZZ_{\ell,X}$ of
$\End_{\Gal(L)}(V_{\ell}(X))$ lies in
$$\ZZ\otimes\Q_{\ell}\subset\End_L(X)\otimes\Q_{\ell}.$$
\item[(ii)]
For all but finitely many primes $\ell$ the center
$\bar{\ZZ}_{\ell,X}$ of $\End_{\Gal(L)}(X_{\ell})$ lies in
$$\ZZ/\ell\subset\End_L(X)\otimes\Z/\ell.$$
 \end{itemize}
 \end{thm}

 \begin{rem}
\label{centerendo} Clearly, for all $\ell$ the commutative
$\Q_{\ell}$-algebra $\ZZ\otimes\Q_{\ell}$ coincides with the center
of $\End_L(X)\otimes\Q_{\ell}$. It is also clear that for all but
finitely many primes $\ell$ the commutative $\F_{\ell}$-algebra
$\ZZ/\ell$ coincides with the center of $\End_L(X)\otimes\Z/\ell\Z$.
Notice also that
$$\End_L(X)\otimes\Q_{\ell}\subset \End_{\Gal(L)}(V_{\ell}(X)), \ \End_L(X)\otimes\Z/\ell\Z\subset \End_{\Gal(L)}(X_{\ell}).$$
This implies that for all $\ell\ne \fchar(L)$
$$\ZZ_{\ell,X}\bigcap [\End_L(X)\otimes\Q_{\ell}]\subset
\ZZ\otimes\Q_{\ell}$$
and for all but finitely many $\ell$
$$\bar{\ZZ}_{\ell,X}\bigcap [\End_L(X)\otimes\Z/\ell]\subset
\ZZ/\ell.$$
 It follows that in order to prove Theorem
\ref{center}, it suffices to check that for all $\ell\ne \fchar(L)$
$$\ZZ_{\ell,X}\subset \End_L(X)\otimes\Q_{\ell}$$
and for all but finitely many $\ell$
$$\bar{\ZZ}_{\ell,X}\subset\End_L(X)\otimes\Z/\ell.$$

 \end{rem}

\begin{rem}
Compare Theorem \ref{center} with \cite[Cor. 4.2.8(ii)]{Deligne}.
\end{rem}

\begin{thm}
\label{ssL}

Let $p>2$ be a prime, $L$ a {\sl geometric field} of
 characteristic $p$ and $X$  an abelian variety  of positive dimension over
 $L$. Then:

\begin{itemize}
 \item[(i)]
For all primes $\ell \ne \fchar(L)$ the  $\Gal(L)$-module
$V_{\ell}(X)$ is semisimple.
\item[(ii)]
For all but finitely many primes $\ell$ the $\Gal(L)$-module
$X_{\ell}$ is semisimple.
 \end{itemize}
\end{thm}

\begin{ex}[Counterexample]
Let $K$ be a field of characteristic $p>2$ that is finitely
generated over a finite field $\F$ and $L=K\bar{\F}$.

 Suppose that $X$ is a {\sl non-supersingular} abelian variety of
positive dimension over $K$ that is actually defined  with all its
endomorphisms over  $\F$. (E.g., one may take as $X$ an {\sl
ordinary} elliptic curve over $\F$.) Then all the torsion points of
$X$ are defined over $\bar{\F}\subset L$. It follows that $\Gal(L)$
acts trivially on all  $X_{n}$ and $V_{\ell}(X)$. In particular,
$$\End_{\Gal)L)}(V_{\ell}(X))=\End_{\Q_{\ell}}(V_{\ell}(X))$$
has $\Q_{\ell}$-dimension $[2\dim(X)]^2$. However, the
$\Q$-dimension of $\End^0(X)$ is strictly less than $[2\dim(X)]^2$
\cite[Lemma 3.1]{ZarhinMRL1} and therefore the centralizer of
$\Gal(L)$ in $\End_{\Q_{\ell}}(V_{\ell}(X))$ is strictly bigger than
$$\End^0(X)\otimes_{\Q}\Q_{\ell}=\End(X)\otimes\Q_{\ell}=\End_L(X)\otimes\Q_{\ell}.$$
This implies that the analogue of the Tate conjecture does not hold
for such $X$ over $L$.
\end{ex}

The paper is organized as follows. In Section \ref{isogenyF} we
discuss a variant of the isogeny theorem with finite coefficients.
Section \ref{repT} contains auxiliary results from representation
theory of groups with procyclic quotients. We prove the main results
in Section \ref{proofM}.


\section{Isogeny theorem with finite coefficients}
\label{isogenyF}

\begin{thm}
\label{isogfinite} Let $K$ be a field that is finitely generated
over its prime subfield and $\fchar(K)\ne 2$. Let $X$ and $Y$ be
abelian varieties over $K$. Suppose that for infinitely many primes
$\ell$ the $\Gal(K)$-modules $X_{\ell}$ and $Y_{\ell}$ are
isomorphic. Then $X$ and $Y$ are isogenous over $K$.
\end{thm}

\begin{proof}
We may assume that $\dim(X)>0$ and $\dim(Y)>0$. Since for all primes
$\ell \ne \fchar(K)$
$$2\dim(X)=\dim_{\F_{\ell}}(X_{\ell}), \
2\dim(Y)=\dim_{\F_{\ell}}(Y_{\ell}),$$ we obtained that
$\dim(X)=\dim(Y)$. Since for all but finitely many primes $\ell$
$$\Hom_K(X,Y)\otimes \Z/\ell\Z=\Hom_{\Gal(K)}(X_{\ell},Y_{\ell}),$$
there exist a prime $\ell \ne \fchar(K)$ and a $K$-homomorphism $u:
X \to Y$ such that $u$ induces an isomorphism between $X_{\ell}$ and
$Y_{\ell}$. In particular, $\ker(u)$ does not contain points of
order $\ell$ on $X$ while the image $u(X)$ contains all points of
order $\ell$ on $Y$. This implies that $\ker(u)$ has dimension zero
while irreducible closed $u(Y)$ has dimension $\dim(Y)$. In other
words, $u: X\to Y$ is a surjective homomorphism with finite kernel,
i.e., is an isogeny.
\end{proof}

\begin{rem}
It would be interesting to get an analogue of Theorem
\ref{isogfinite}  where say, a number field $K$ is replaced by its
infinite $\ell$-cyclotomic extension $K(\mu_{\ell^{\infty}})$. Some
important  special cases of this analogue are done in \cite{HM}.

\end{rem}

\section{Representation theory}
\label{repT}

Throughout this Section, $G$ is a profinite group and $H$ is a
closed normal subgroup of $G$ such that the quotient $\Gamma=G/H$ is
a procyclic group. We call $G$ a {\sl procyclic extension} of $H$.

We write down the group operation in $G$ (and $H$) multiplicatively
and in $\Gamma$ additively. We write $\pi:G \to \Gamma$ for the
natural continuous surjective homomorphism from $G$ to $\Gamma$.  If
$n$ is a positive integer then $n\Gamma$ is the closed subgroup (as
the image of compact $\Gamma$ under $\Gamma \stackrel{n}{\to}
\Gamma$) in $\Gamma$, whose index divides $n$; since the index is
finite, $n\Gamma$ is open in $\Gamma$. Notice that $n\Gamma$ is also
a procyclic group.

Let us put $G_n=\pi^{-1}(n\Gamma)$; clearly, $G_n$ is an open normal
subgroup in $G$, whose index divides $n$. In addition,  each $G_n$
contains $H$ and the quotient $G/G_n$ is canonically isomorphic to
$\Gamma/n\Gamma$  while $G_n/H \cong n\Gamma$. In particular, $H$ is
a closed normal subgroup of $G_n$ and the quotient $G_n/H$ is
procyclic, i.e. $G_n$ is also a procyclic extension of $H$. In
particular, for each positive integer $m$ we may define the open
normal subgroup $(G_n)_m$ of $G_n$; clearly,
$$(G_n)_m=G_{mn},$$
because $m(n\Gamma)=(mn)\Gamma$.

\begin{rem}
\label{finiteorder} Let $c: G \to k^{*}$ be a continuous group
homomorphism (character) of $G$ with values in the multiplicative
group of a locally compact field $k$ that enjoys the following
properties:

\begin{itemize}
\item[(i)] $c$ kills $H$, i.e., $c$ factors through $G/H=\Gamma$.
\item[(i)] $c^n$ is the trivial character, i.e., $c^n$ kills the
whole $G$.
\end{itemize}
Then obviously $c$ kills $\pi^{-1}(n\Gamma)=G_n$, i.e., $c$ factors
through the finite cyclic quotient $G/G_n=\Gamma/n\Gamma$.
\end{rem}


Let $k$ be a locally compact field (e.g., $k$ is finite or
$\Q_{\ell}$.) Let $d$ be a positive number and $V$ a $d$-dimensional
$k$-vector space provided with  natural topology induced by the
topology on $k$. Let
$$\rho: G \to \Aut_k(V)\cong \GL(d,k)$$
be a continuous semisimple linear representation of $G$. As usual,
$\det(V)$ stands for the one-dimensional $G$-module
$\Lambda^d_{k}(V)$.

\begin{lem}
\label{absimple} Suppose that
$$\End_{G_d}(V)=k.$$
 Then the $H$-module $V$ is absolutely simple. In particular,
$$\End_H(V)=k.$$
\end{lem}

\begin{rem}
Lemma \ref{absimple} asserts that if $W$ is an absolutely simple
$G_d$-module then it remains absolutely simple, being viewed as a
$H$-module.
\end{rem}

\begin{proof}
We have
$$k\subset \End_{G}(V)\subset \End_{G_d}(V)\subset \End_{H}(V).$$
Since $\End_{G_d}(V)=k$, we conclude that $k= \End_{G}(V)$.

By Clifford's Lemma \cite[Theorem (49.2)]{CR}, the $H$-module $V$ is
semisimple. Let us split $V$ into a direct sum $V=\oplus_{i=1}^r
V_i$ of isotypic $H$-modules. Clearly $G$ permutes $V_i$'s; the
simplicity of the $G$-module $V$ implies that $G$ acts on $\{V_1,
\dots , V_r\}$ transitively. In particular, all $V_i$'s have the
same dimension and therefore
$$\dim(V_i)=\frac{\dim(V)}{r}=\frac{d}{r};$$
in particular, $r\mid d$. Clearly, the action of $G$ on $\{V_1,
\dots , V_r\}$ factors through $G/H$. Since this action is
transitive and  $G/H$ is procyclic, this action factors through
finite cyclic $G/G_r$ and therefore through $G/G_d$, i.e,, each
$V_i$ is a $G_d$-submodule. Since the $G_n$-module $V$ is
(absolutely) simple, $V=V_i$. In other words, the $H$-module $V$ is
isotypic.Then the centralizer
$$D=\End_H(V)$$
is a simple $k$-algebra. Let $k^{\prime}$ be the center of $D$: it
is an overfield of $k$. Clearly, $V$ becomes a $k^{\prime}$-vector
space; in particular, $k^{\prime}/k$ is a finite algebraic extension
and $[k^{\prime}:k]\mid d$. On the other hand, since $H$ is normal
in $G$,
$$\rho(g) D \rho(g)^{-1}=D \ \forall g\in G.$$
Clearly, the center $k^{\prime}$ is also stable under the
conjugations by elements of $\rho(G)$ and $\{k^{\prime}\}^{G}=k$.
This gives us a continuous group homomorphism $G/H \to
\Aut(k^{\prime}/k)$ such that $\{k^{\prime}\}^{G/H}=k$. It follows
that $k^{\prime}/k$ is a finite cyclic Galois extension and $$G/H
\to \Aut(k^{\prime}/k)=\Gal(k^{\prime}/k)$$ is a surjective
homomorphism. Since $\#(\Gal(k^{\prime}/k))=[k^{\prime}:k]$ divides
$d$, the surjection $\Gal(k^{\prime}/k)\twoheadrightarrow
\Gal(k^{\prime}/k)$ factors through $G/G_d$ and therefore
$$k^{\prime} \subset \End_{G_d}(V);$$
since $\End_{G_d}(V)=k$, we conclude that $k^{\prime}\subset k$ and
therefore  $k^{\prime}= k$. This means that $D$ is a central simple
$k$-algebra and let $t:=\dim_k(D)$. We need to prove that $t=1$.
Suppose that $t>1$, pick a generator in $\Gamma$ and denote by $g$
its preimage in $G$. Then the map
$$u \mapsto \rho(g) u \rho(g)^{-1}$$
is an automorphism of $D$, whose set of fixed points coincides with
$k$. By Skolem-Noether theorem, there exists an element $z \in
D^{*}$ such that
$$\rho(g) u \rho(g)^{-1}=zu z^{-1} \ \forall u \in D.$$
Clearly, $z$ itself is a fixed point of this automorphism and
therefore $z\in k$, which implies that the automorphism is the
identity map and therefore its set of fixed points must be the whole
$D$, which is not the case, because $t>1$. The obtained
contradiction proves that $t=1$, i.e.,
$$\End_H(V)=D=k$$
and we are done.
\end{proof}

\begin{lem}
\label{twist} Let $\rho_1:G\to \Aut_k(W_1)$  be a continuous linear
$d$-dimensional representation of $G$ over $k$. Let $\rho_2:G\to
\Aut_k(W_2)$ be a linear finite-dimensional continuous
representation of $G$ over $k$. Suppose that $\End_H(W_1)=k$ and the
$H$-modules $W_1$ and $W_2$ are isomorphic. Then there exists a
continuous character
$$\chi: G/H=\Gamma \to k^{*}$$
such that the $G$-module $W_2$ is isomorphic to the twist
$W_1(\chi)$. In particular, the one-dimensional $G$-modules
$\det(W_2)$ and $[\det(W_1)](\chi^d)$ are isomorphic.
\end{lem}

\begin{proof}
It is well known that the vector space $\Hom_k(W_1,W_2)$ carries the
natural structure of a $G$-module defined by
$$g: u \mapsto \rho_2(g)u\rho_1(g)^{-1} \ \forall g\in G, \ u \in \Hom_k(W_1,W_2).$$
Since $H$ is normal in $G$, the subspace $\Hom_H(W_1,W_2)$ of
$H$-invariants is a $G$-invariant subspace in $\Hom_k(W_1,W_2)$. Our
conditions on the $H$-module $W_1$ and $W_2$ imply that the
$k$-vector space $\Hom_H(W_1,W_2)$ is one-dimensional (and its every
nonzero element $W_1\to W_2$ is an isomorphism of $H$-modules).
Therefore the action of $G$ on one-dimensional $\Hom_H(W_1,W_2)$ is
defined by a certain continuous character $\chi: G \to k^{*}$, which
obviously kills $H$, so we may view $\chi$ as a continuous character
$$\Gamma=G/H \to k^{*}.$$
This means that if $u: W_1 \cong W_2$ is an isomorphism of
$H$-modules then
$$\rho_2(g)u\rho_1(g)^{-1}=\chi(g) u \ \forall g\in G.$$
Multiplying this equality from the right by $\rho_1(g)$, we obtain
that
$$\rho_2(g)u=\chi(g) u \rho_1(g)= u [\chi(g)\rho_1(g)]\ \forall g\in G,$$
which means that $u$ is an isomorphism of $G$-modules $W_1(\chi)$
and $W_2$. It remains to notice that
$\det(W_1(\chi))=[\det(W_1)](\chi^d)$.
\end{proof}

\begin{cor}
We keep the notation and assumptions of Lemma \ref{twist}. If for
some positive integer $N$ the $G$-modules $[\det(W_1)]^{\otimes N}$
and $[\det(W_2)]^{\otimes N}$ are isomorphic then the character
$\chi^{Nd}$ is trivial.
\end{cor}

\begin{thm}
\label{centerDel} Suppose that the $G$-module $V$ is semisimple.
Then there exists a positive integer $n$ that  depends only on $d$
and such that the center of $\End_H(V)$ lies in $\End_{G_n}(V)$.
\end{thm}

\begin{proof}
By a variant of Clifford's Lemma \cite[Lemma 3.4]{ZarhinMatSb}, the
$H$-module $V$ is semisimple. In particular, the centralizer
$D=\End_H(V)$ is a (finite-dimensional) semisimple $k$-algebra.
Since $H$ is normal in $G$
$$\rho(g)D \rho(g)^{-1}=D \ \forall g\in G.$$
Let $Z$ be the center of $D$. Since $D$ is semisimple, $Z$ is
isomorphic to a direct sum $\oplus_{i=1}^r k_i$ of finitely many
overfields $k_i\supset k$ where each $k_i/k$ is a finite algebraic
field extension. Clearly,
 $$[k_i:k]\le \dim_k(Z)\le \dim_k(V)=d, \ r\le d$$ and the
$k$-algebra $Z$ has exactly $r$ minimal idempotents (the identity
elements $e_i$'s of $k_i$'s. Clearly, group $\Aut_k(Z)$ of
$k$-linear automorphisms of $Z$ permutes $e_i$'s, which gives us the
homomorphism from $\Aut_k(Z)$ to the full symmetric group $\ST_r$,
whose kernel leaves invariant each summand $k_i$ and therefore sits
in the product $\prod_{i=1}^r \Aut(k_i/k)$, whose order does not
exceed  $\prod_{i=1}^r [k_i:k]\le d^d$. It follows that $\Aut_k(Z)$
is a finite group, whose order  does not exceed $d!\cdot d^d$. This
implies that $n:=(d! \cdot d^d)!$  is divisible by the order of
$\Aut_k(Z)$.

 On the
other hand, clearly,
$$\rho(g)Z \rho(g)^{-1}=Z \ \forall g\in G,$$
because every automorphism of $D$ respects its center. This gives us
the group homomorphism
$$\phi: G \to \Aut_k(Z), \ \phi(g)(z)=\rho(g)z \rho(g)^{-1} \ \forall
z\in Z, \ g\in G,$$ which kills $H$, because $$Z\subset
D=\End_H(V).$$ Clearly, $\phi$ kills $G_n$ and we are done.
\end{proof}

\section{Proofs of main results}
\label{proofM}  There is a subfield $K\subset L$ such that $K$ is
finitely generated over $\F=\F_p$ and the compositum $K\bar{\F}=L$
while given abelian varieties $X$ and $Y$, their group laws and
zeros are defined over $K$. We also require that
$$\End_K(X)=\End_L(X), \ \End_K(Y)=\End_L(Y)\eqno(4).$$
Let us put
$$G=\Gal(K), H=\Gal(L), \Gamma=\Gal(L/K).$$
Since $\bar{\F}/\F_p$ is a Galois extension and
  $K\bar{\F}=L$, the Galois group $\Gamma=\Gal(L/K)$ is
canonically isomorphic to a closed subgroup of
$\Gal(\bar{\F}/\F_p)$; since the latter is procyclic, $\Gamma$ is
also procyclic.

Let $n$ be a positive integer and let us consider the open normal
subgroup $G_{n}$ of $G$. Since $G_n$ contains $H$, the subfield
$K_{n}={\bar{K}_s}^{G_{n}}$ of $G_{n}$-invariants is a finite
(cyclic) Galois extension of $K$ that lies in $L$. In particular,
$K_{n}$ is finitely generated over $\F_p$ and $\Gal(K_{n})=G_{n}$.
Since $K\subset K_n \subset L$, it follows from (4) that
$$\End_{K_n}(X)=\End_L(X), \ \End_{K_n}(Y)=\End_L(Y)\eqno(5).$$
 If $\ell$ is a prime different from $p$ we write
$$\bar{\chi}_{\ell}:\Gal(K)\to (\Z/\ell\Z)^{*}=\F_{\ell}^{*}, \ {\chi}_{\ell}:\Gal(K)\to
\Z_{\ell}^{*}\subset \Q_{\ell}^{*}$$ for the cyclotomic characters
that define the Galois action on all $\ell$th roots of unity (resp.
all $\ell$-power roots of unity). Clearly,
$$\bar{\chi}_{\ell}=\chi_{\ell}\bmod\ell \eqno(6).$$

Since $K$ is finitely generated over $\F_p$, the cyclotomic
characters enjoy the following properties:

\begin{itemize}
\item[(i)]
The character ${\chi}_{\ell}$ has infinite multiplicative order.
\item[(ii)]
If $N$ is a positive integer then for all but finitely many primes
$\ell$ the character $\bar{\chi}_{\ell}^N$ is nontrivial.
\end{itemize}

Since every $K_n$ is finitely generated over $\F_p$, the abelian
variety $X$ over $K$ enjoys the following properties.

\begin{itemize}
\item[(a)]
For all primes $\ell \ne \fchar(K)$ the $G_n$-module $V_{\ell}(X)$
is semisimple and
$$\End_{G_n}(V_{\ell}(X))=\End_{K_n}(X)\otimes\Q_{\ell}=
\End_L(X)\otimes\Q_{\ell}.$$ In particular, if $\End_L(X)=\Z$ then
$G_n$-module $V_{\ell}(X)$ is absolutely simple.
\item[(b)]
For all but finitely many primes $\ell$ the $G_n$-module $X_{\ell}$
is semisimple and
$$\End_{G_n}(X_{\ell})=\End_{K_n}(X)\otimes\Z/\ell=
\End_L(X)\otimes\Z/\ell.$$ In particular, if $\End_L(X)=\Z$ then
$G_n$-module $X_{\ell}$ is absolutely simple for all but finitely
many primes $\ell$.

\end{itemize}

\begin{proof}[Proof of Theorem \ref{main}]
Let $d=\dim(X)$. Let us consider the open normal subgroup $G_{2d}$
of $G$.

Since $\End_L(X)=\Z$, (a) tells us that  the $G_{2d}$-module
$V_{\ell}(X)$ is absolutely simple for each $\ell \ne p$; in
particular,
$$\Q_{\ell}=\End_{G_{2d}}(V_{\ell}(X))=\End_{G}(V_{\ell}(X)).$$
On the other hand, (b) tells us that   for all but finitely many
$\ell$ the $G_{2d}$-module $X_{\ell}$ is absolutely simple; in
particular,
$$\F_{\ell}=\End_{G_{2d}}(X_{\ell})=\End_{G}(X_{\ell}).$$

Now, in order to finish the proof of Theorem \ref{main} it suffices
to apply Lemma \ref{absimple} in the following situations (taking
into account that
$2d=\dim_{\Q_{\ell}}(V_{\ell}(X))=\dim_{\F_{\ell}}(X_{\ell})$).

\begin{itemize}
\item[(i)]
$k=\Q_{\ell}, \ V=V_{\ell}(X)$.
\item[(ii)]
$k=\F_{\ell}, \ V=X_{\ell}$.
\end{itemize}
\end{proof}

\begin{proof}[Proof of Theorem \ref{isogeny}]
Clearly, $d:=\dim(X)=\dim(Y)$. It is well known that the existence
of Galois-equivariant nondegenerate alternating bilinear
(Weil--Riemann) forms on Tate modules \cite[ Sect. 1.3]{SerreIzv},
\cite[Proof of Prop. 2.2]{ZarhinMatSb} implies that
 $\det(V_{\ell}(X))$ and
$\det(V_{\ell}(Y))$ are one-dimensional $G$-modules defined by the
character $\chi_{\ell}^d$. Now applying Lemma \ref{twist}, we
conclude the $G$-module $V_{\ell}(Y)$ is isomorphic to the twist
$V_{\ell}(X)(\chi)$ for a certain continuous character
$\chi:G/H=\Gamma\to \Q_{\ell}^{*}$. It follows from Corollary to
Lemma \ref{twist} that $\chi^{2d}$ is trivial. This implies that
$\chi$ kills $G_{2d}$ and therefore the $G_{2d}$-modules
$V_{\ell}(X)$ and $V_{\ell}(Y)$ are isomorphic. Now the isogeny
theorem over $K_{2d}$ implies that $X$ and $Y$ are isogenous over
$K_{2d}$ and therefore over $L$. This proves (i).

Similar arguments work in the case (ii). Clearly,
$d:=\dim(X)=\dim(Y)$ and the structure of $\Gal(K)$-modules on the
rank $1$ free $\Z_{\ell}$-modules
$\Lambda^{2d}_{\Z_{\ell}}T_{\ell}(X)$ and
$\Lambda^{2d}_{\Z_{\ell}}T_{\ell}(Y)$ is defined by $\chi_{\ell}^d$,
because
$$\Lambda^{2d}_{\Z_{\ell}}T_{\ell}(X)\subset
\Lambda^{2d}_{\Q_{\ell}}V_{\ell}(X)=\det(V_{\ell}(X)), \
\Lambda^{2d}_{\Z_{\ell}}T_{\ell}(Y)\subset
\Lambda^{2d}_{\Q_{\ell}}V_{\ell}(Y)=\det(V_{\ell}(Y)).$$ It follows
from (0) that
$$\det(X_{\ell})=\Lambda^{2d}_{\Z_{\ell}}X_{\ell}=\Lambda^{2d}_{\Z_{\ell}}(T_{\ell}(X)/\ell)=[\Lambda^{2d}_{\Z_{\ell}}(T_{\ell}(X)]/\ell
$$ and therefore the structure of
the Galois module on $\det(X_{\ell})$ is defined by the character
$\chi_{\ell}^d\bmod \ell=\bar{\chi}_{\ell}^d$. By the same token,
the structure of the Galois module on the one-dimensional
$\det(Y_{\ell})$ is also defined by $\bar{\chi}_{\ell}^d$. Now
applying Lemma \ref{twist}, we conclude the $G$-module $Y_{\ell}$ is
isomorphic to the twist $Y_{\ell}(\bar{\chi})$ for a certain
continuous character $\bar{\chi}:G/H=\Gamma\to \F_{\ell}^{*}$. It
follows from Corollary to Lemma \ref{twist} that $\bar{\chi}^{2d}$
is trivial. As above, this implies that $\bar{\chi}$ kills $G_{2d}$
and therefore the $G_{2d}$-modules $X_{\ell}$ and $Y_{\ell}$ are
isomorphic for infinitely many $\ell$. Now Theorem \ref{isogfinite}
implies that $X$ and $Y$ are isogenous over $K_{2d}$ and therefore
over $L$. This proves (ii).
\end{proof}

\begin{proof}[Proof of Theorem \ref{center}]
As above, $G=\Gal(K), H=\Gal(L)$.

 (i) Let us put $k=\Q_{\ell}, V=V_{\ell}(X)$ and apply
Theorem \ref{centerDel}. We obtain that there exists a positive
integer $n$ such that the center of $\End_{\Gal(L)}(V_{\ell}(X))$
lies in $\End_{G_n}(V_{\ell}(X))\otimes\Q_{\ell}$. By (a),
$$\End_{G_n}(V_{\ell}(X))=\End_{K_n}(X)\otimes\Q_{\ell}=
\End_{L}(X)\otimes\Q_{\ell}$$ and we are done.

(ii) Let us put $k=\F_{\ell}, V=X_{\ell}$ and apply Theorem
\ref{centerDel}. We obtain that there exists a universal positive
integer $n$ that depends only on $2\dim(X)$ such that for all but
finitely many primes $\ell$ the center of $\End_{\Gal(L)}(X_{\ell})$
lies in $\End_{G_n}(X_{\ell})$. By (b),
$$\End_{G_n}(X_{\ell})=\End_{K_n}(X)\otimes\Z/\ell=
\End_{L}(X)\otimes\Z/{\ell}$$ and we are done, taking into account
Remark \ref{centerendo}.
\end{proof}

\begin{proof}[Proof of Theorem \ref{ssL}]
Recall that  $H=\Gal(L)$ is a {\sl normal} subgroup in $G=\Gal(K)$.
By the variant of Clifford's lemma \cite[Lemma 3.4]{ZarhinMatSb},
the semisimplicity of the $\Gal(K)$-modules $V_{\ell}(X)$ and
$X_{\ell}$ implies that they are semisimple $\Gal(L)$-modules.
\end{proof}

\end{document}